\documentclass[10pt, one side,emlines]{amsart}
\usepackage{amssymb,latexsym,xy,eucal,mathrsfs,tikz}
\textwidth=15.2cm \textheight=24cm \theoremstyle{plain}
\newtheorem{lem}{Lemma}[section]

\newtheorem{thm}[lem]{Theorem}

\newtheorem{cor}[lem]{Corollary}
\theoremstyle{definition}
\newtheorem{exa}[lem]{Example}
\newtheorem{rem}[lem]{Remark}
\newtheorem{obs}[lem]{Observation}
\newtheorem{con}[lem]{Construction}
\newtheorem{defn}[lem]{Definition}
\newtheorem*{prob}{Problem}

\newtheorem{st}[lem]{Structure Theorem}
\newtheorem{ilu}[lem]{Illustration}
\theoremstyle{remark}

\numberwithin{equation}{section}  \voffset
-55truept \hoffset -15truept

\begin{document}
\baselineskip 15truept

\hfill{\it Accepted for publication in Mathematica Slovaca}\vspace{.21in}

\title{Zero-divisor graphs of lower dismantlable lattices-II}
\date{}
\author[Avinash Patil, B. N. Waphare \and Vinayak Joshi ] %
{Avinash Patil*, B. N. Waphare** \and  Vinayak Joshi**}

\newcommand{\acr}{\newline\indent}

\address{\llap{*\,}Department of Mathematics\acr
                   Garware College of Commerce\acr
                   Karve Road, Pune-411004\acr
                   India.}
\email{avipmj@gmail.com}

\address{\llap{**\,}Center For Advanced Studies In Mathematics\acr Department of Mathematics\acr
                   Savitribai Phule Pune University\acr
                   Pune-411007\acr
                   India.}

\email{bnwaph@math.unipune.ac.in; waphare@yahoo.com}
\email{vvj@math.unipune.ac.in; vinayakjoshi111@yahoo.com}

\begin{abstract}In this paper, we continue our study of the zero-divisor graphs of lower dismantlable lattices that was started in \cite{ap}. The present paper mainly deals with an Isomorphism Problem for the zero-divisor graphs of lattices. In fact, we prove that the zero-divisor graphs of lower dismantlable lattices with the greatest element 1 as join-reducible are isomorphic if and only if the lattices are isomorphic.
\end{abstract}
\maketitle
{\bf MSC(2010):} {Primary: $05$C$60$, Secondary: $05$C$25$}

\section{Introduction}
Beck \cite{2} introduced the concept  of zero-divisor graph of a
commutative ring $R$ with unity as follows. Let $G$ be a simple
graph whose vertices are the elements of $R$ and two vertices $x$
and $y$ are adjacent if $xy = 0$. The graph $G$ is known as the
\textit{zero-divisor graph} of $R$. He was mainly interested in
the coloring of this graph. This concept is well studied in
algebraic structures such as  semigroups, rings, lattices,
semi-lattices as well as in ordered structures such as posets and
qosets; see Alizadeh et al. \cite{1}, Anderson et al. \cite{DP},
Hala\v {s} and Jukl \cite{4}, Hala\v {s} and L\"{a}nger \cite{5},
 Joshi \cite{7}, Joshi and Khiste \cite{8,vjau},   Joshi, Waphare and Pourali
\cite{9, jwp}, Lu and Wu \cite{13} and
Nimbhorkar et al. \cite{15}.

It is easy to observe that if two posets $P_1$ and $P_2$ are isomorphic then their zero-divisor graphs $G_{\{0\}}(P_1)$ and $G_{\{0\}}(P_2)$ are isomorphic.
But the converse need not be true in general. Hence it is worth to study the following Isomorphism Problem.

\vspace{.5cm}
\noindent
{\bf Isomorphism Problem: Find a class of posets $\mathbf{\mathbb{P}}$ such that
$\mathbf{G_{\{0\}}(P_1)\cong G_{\{0\}}(P_2)}$ if and only if $\mathbf{P_1\cong P_2}$
for $\mathbf{P_1,P_2\in \mathbb{P}}$.}
\vspace{.5cm}

Joshi and Khiste \cite{8} solved Isomorphism Problem for Boolean posets, which essentially extends the result of LaGrange \cite{11}, see also Mohammadian \cite{14}.
 Recently Joshi, Waphare and Pourali \cite{9} solved this problem for the class of  section semi-complemented(SSC) meet semi-lattices.

In the sequel, we obtain a solution for the class of lower dismantlable lattices, a subclass of dismantlable lattices.

In this paper, we continue our study of the zero-divisor graphs of lower dismantlable lattices that was started in \cite{ap}. The present paper mainly deals with an Isomorphism Problem of zero-divisor graphs for lower dismantlable lattices.  In fact we prove:
\begin{thm}\label{t1} Let  $\mathscr{L}$ be the class of lower dismantlable lattices with the greatest element 1 as join-reducible. Then $G_{\{0\}}(L_1)\cong G_{\{0\}}(L_2)$ if and only if $L_1\cong L_2$ for $L_1,L_2\in \mathscr{L}$.\end{thm}

Rival \cite{17} introduced dismantlable lattices to study the
combinatorial properties of doubly irreducible elements. By a 
dismantlable lattice, we mean a lattice which can be completely
``dismantled" by removing one element at each stage. Kelly and Rival
\cite{10} characterized dismantlable lattices by means of crowns,
whereas Thakare, Pawar and Waphare \cite{18} gave a structure
theorem for dismantlable lattices using adjunct operation.

Now, we begin with the necessary definitions and terminology.

 First, we define the covering relation. We say that `{\it $a$ is covered by $b$}', if there is no $c$ such that $a<c<b$ and we denote it by $a\prec b$. Further, $a$ is a lower cover of $b$ and $b$ is an upper cover of $a$.
\begin{defn}[Thakare et al. \cite{18}]\label{d1}
Let $L_1$ and $L_2$ be two disjoint finite lattices and $\; (a, b)\; $ is a pair
of elements in $L_1$ such that $a < b$ and $a \not\prec b$. Define the
partial order $\leq$ on the set $L = L_1 \cup L_2$ with respect to
the pair $(a,b)$ as follows. For $x,y\in L$, we say $x \leq y$ in $L$ if

 either $x,y \in L_1$ and $x \leq y$ in $L_1$; \hspace{1.32cm} or $x,y \in L_2$ and $x \leq y$ in $L_2$;

or $x \in L_1,$ $ y \in L_2$ and $x \leq a$ in $L_1$; \hspace{1cm}
or $x \in L_2,$ $ y \in L_1$ and $b \leq y$ in $L_1$.\\
Notice that $L$ is a lattice containing $L_1$ and $L_2$
as sublattices. The procedure of obtaining $L$ in this way is
called an {\it adjunct operation of $L_2$ to $L_1$}. The pair $(a,b)$
is called {\it an adjunct pair } and $L$ is an {\it adjunct }
of $L_2$ to $L_1$ with respect to the adjunct pair ($a,b$) and we
write $L = L_1 ]^b_a L_2$ (see Figure \ref{f1}).
\end{defn}
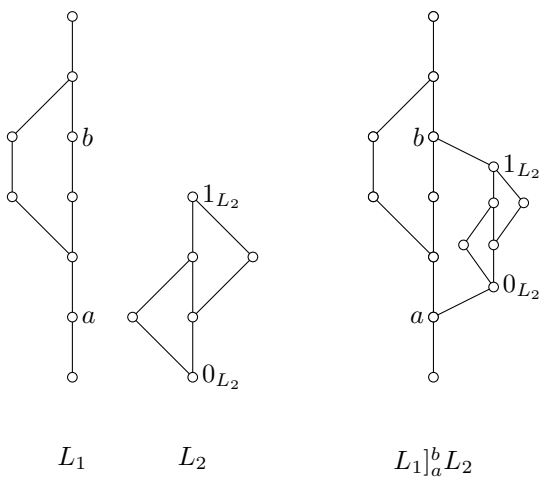
\begin{figure}[h]
\begin{tikzpicture}[scale=.8]
\draw (0,0)--(0,6); \draw (0,2)--(-1,3)--(-1,4)--(0,5);
\draw[fill=white](0,0) circle(.08); \draw[fill=white](0,1)
circle(.08); \draw[fill=white](0,2) circle(.08);
\draw[fill=white](0,3) circle(.08); \draw[fill=white](0,4)
circle(.08); \draw[fill=white](0,5) circle(.08);
\draw[fill=white](0,6) circle(.08); \draw[fill=white](-1,3)
circle(.08); \draw[fill=white](-1,4) circle(.08); \node [right] at
(0,1){$a$}; \node [right] at (0,4){$b$}; \node [below] at
(0,-1){$L_1$};

\draw (2,0)--(2,3); \draw (2,0)--(1,1)--(2,2); \draw
(2,1)--(3,2)--(2,3); \draw[fill=white](2,0) circle(.08);
\draw[fill=white](2,1) circle(.08); \draw[fill=white](2,2)
circle(.08); \draw[fill=white](2,3) circle(.08);
\draw[fill=white](1,1) circle(.08); \draw[fill=white](3,2)
circle(.08); \node [below] at (2,-1){$L_2$}; \node [right] at
(2,0){$0_{\tiny L_2}$};\node [right] at (2,3){$1_{\tiny L_2}$};

\draw (6,0)--(6,6); \draw (6,2)--(5,3)--(5,4)--(6,5);
\draw[fill=white](6,0) circle(.08); \draw[fill=white](6,1)
circle(.08); \draw[fill=white](6,2) circle(.08);
\draw[fill=white](6,3) circle(.08); \draw[fill=white](6,4)
circle(.08); \draw[fill=white](6,5) circle(.08);
\draw[fill=white](6,6) circle(.08); \draw[fill=white](5,3)
circle(.08); \draw[fill=white](5,4) circle(.08);

\draw (6,1)--(7,1.5); \draw (6,4)--(7,3.5); \draw
(7,1.5)--(7,3.5); \draw (7,1.5)--(6.5,2.2)--(7,2.9); \draw
(7,2.2)--(7.5,2.9)--(7,3.5); \draw[fill=white](7,1.5) circle(.08);
\draw[fill=white](7,2.2) circle(.08); \draw[fill=white](7,2.9)
circle(.08); \draw[fill=white](7,3.5) circle(.08);
\draw[fill=white](6.5,2.2) circle(.08); \draw[fill=white](7.5,2.9)
circle(.08); \node [left] at (6,1){$a$}; \node [left] at
(6,4){$b$}; \draw[fill=white](6,1) circle(.08);
\draw[fill=white](6,2) circle(.08); \draw[fill=white](6,3)
circle(.08); \draw[fill=white](6,4) circle(.08);
\draw[fill=white](6,5) circle(.08); \draw[fill=white](6,6)
circle(.08); \draw[fill=white](5,3) circle(.08);
\draw[fill=white](5,4) circle(.08); \node [below] at (6,-1)
{$L_1]_a^bL_2$}; \node [right] at (7,1.5){$0_{\tiny L_2}$};\node
[right] at (7,3.5){$1_{\tiny L_2}$};
\end{tikzpicture}
\caption{Adjunct of two lattices $L_1$ and $L_2$}\label{f1}
\end{figure}
We use the following definition of zero-divisor graph.
\begin{defn}[Joshi \cite{7}]\label{joshi} Let $L$ be a lattice with 0. We associate a simple undirected graph $G_{\{0\}}(L)$
to $L$. The set of  vertices of $G_{\{0\}}(L)$ is
$V(G_{\{0\}}(L))=\big\{x \in L \setminus \{0\} |  x\wedge
y=0$ for some $y \in L\setminus \{0\} \big\}$ and two  distinct
vertices $x, y$ are adjacent if and only if $x\wedge y=0$. The
graph $G_{\{0\}}(L)$ is called the {\it zero-divisor graph} of
$L$.
\end{defn}
 The following result is essentially due to Joshi \cite[Theorem 2.4]{7}.
\begin{thm}[{Joshi \cite[Theorem 2.4]{7}}]\label{403}
	Let $L$ be a lattice with $V(G_{\{0\}}(L))\ne \emptyset$.  Then $G_{\{0\}}(L)$ is connected with $diam(G_{\{0\}}(L)) \leq 3$.
\end{thm}

An element $x$ in a lattice $L$ is \textit{join-reducible $($meet-reducible$)$} in $L$
 if there exist $y,z\in L$ both distinct from $x$, such that $y\vee z=x$ $(y\wedge z= x)$; $x$ is
 \textit{join-irreducible $($meet-irreducible$)$} if it is not join-reducible $($meet-reducible$)$;
 $x$ is \textit{doubly irreducible} if it is both join-irreducible and meet-irreducible.
 Therefore, an element $x$ is doubly irreducible in a finite lattice $L$ if and only if $x$ has just one
 lower cover as well as just one upper cover. The set of all join-irreducible $($meet-irreducible$)$
 elements of $L$ is denoted by \textit{$J(L)$ $(M(L))$}.  From the definitions of join-irreducibility and meet-irreducibility, it is clear that $0\in J(L)$ and $1\in M(L)$.
 The set of all doubly irreducible elements of $L$ is
 denoted by $Irr(L)$ and the set of doubly reducible elements of $L$ is denoted by \textit{$Red(L)$}. Thus, if $x\in Red(L)$ then $x$
 is either join-reducible or meet-reducible. A nonzero element $p$ of a lattice $L$ with 0 is an {\it atom} if $0\prec p$.
 The set of atoms  in a lattice $L$ is denoted by $At(L)$.

The \textit{cover graph} of a lattice $L$, denoted by
     $CG(L)$, is the graph whose vertices are the elements of $L$ and whose edges are the pairs $(x,y)$ with $x,y\in L$
     satisfying $x\prec y$ or $y\prec x$. The edge set of $CG(L)$ is denoted by $E(CG(L))$.  The \textit{comparability graph} of a lattice $L$, denoted by $C(L)$, is the graph whose
     vertices are the elements of $L$ and two vertices $x$ and $y$ are adjacent if and only if $x$ and $y$ are comparable.
     The complement of the comparability graph,  $C(L)^c$, is called the \textit{incomparability graph} of $L$.

 Let $x$ be a vertex of a graph $G$. The set of neighbors of $x$ in $G$, denoted by $N(x)$, is given  by
 $\{y\in V(G)|~ x \textnormal{ and }y \textnormal{ are adjacent in } G\}$. The relation defined by $x\sim y$
 if and only if $N(x)=N(y)$ is an equivalence relation on $V(G)$. The equivalence class $[x]$ of $G$ is given by $[x]=\{y\in V(G)|y\sim x\}$. For undefined notions and terminology  from Graph Theory, the reader is referred to West \cite{6}.

\begin{defn}[Rival \cite{17}]A finite lattice $L$ having $n$ elements is {\it dismantlable},
if there exists a chain $L_1 \subset L_2 \subset \cdots \subset
L_n (= L)$ of sublattices of $L$ such that $| L_i| = i$, for all
$i$.\end{defn}

\begin{st}[{Thakare et al. \cite[Theorem 2.2]{18}}]\label{76}A lattice is  dismantlable if and only if it is an adjunct of chains.\end{st}

Thus a dismantlable lattice is of the form $L=(\cdots((C_0]_{a_1}^{b_1}C_1)]_{a_2}^{b_2}C_2)\cdots)]_{a_r}^{b_r} C_r$, {\it i.e.},
 start with the chain $C_0$ adjoin $C_1$ between the pair $(a_1,b_1)$ and so on. If the context is clear,
 from now onwards, we will simply write $L=C_0]_{a_1}^{b_1}C_1]_{a_2}^{b_2}\cdots]_{a_r}^{b_r}C_r$.

Thakare, Pawar and Waphare  \cite{18} proved that a dismantlable lattice $L$ need not have a unique adjunct representation but an adjunct pair $(a, b)$ occurs the same number of times in any adjunct representation of $L$.
\begin{thm}[{Thakare et al. \cite[Theorem 2.7]{18}}]A pair $(a, b)$ occurs $r$ times in an adjunct representation of
a dismantlable lattice $L$ if and only if there exist exactly $r + 1$ maximal chains
$C_0',C_1',\cdots, C_r'$ in $[a, b]$ such that $x\wedge y = a$ and $x\vee y = b$ for any
$x\in  C_i'\backslash \{a,b\}$, $y\in C_j'\backslash \{a,b\}$,  $i\neq j$.
\end{thm}

Now, we recall the definition of a lower dismantlable lattice and
its properties from \cite{ap}.

\begin{defn}We call a dismantlable lattice $L$ to be a \textit{lower dismantlable}, if it is a chain or every adjunct pair in $L$ is of the form $(0,b)$ for some $b\in L$.
\end{defn}
The adjunct representation of a lower dismantlable lattice is of the type
$L=C_0]_0^{x_1}C_1]_0^{x_2}\cdots]_0^{x_n}C_n$, where $C_i$'s are chains. We call an element $x$ of a lower
dismantlable lattice $L$ an {\it adjunct element} if $(0,x)$ is an adjunct pair in $L$.

\begin{rem}\label{r1}
In a lower dismantlable lattice there is no nonzero meet-reducible element.
\end{rem}

 The following lemma is proved in \cite{ap} and gives the properties of lower dismantlable
 lattices which will be used frequently in sequel.
\begin{lem} \label{400}Let $L=C_0]_0^{x_1}C_1]_0^{x_2}\cdots]_0^{x_n}C_n$ be a lower dismantlable lattice, where $C_i$'s are chains. Then for nonzero elements $a,b\in L$,
the following statements are true.
\begin{enumerate}
\item[$a)$] $a\wedge b=0$ if and only if $a||b$ (where $a||b$ means $a$ and $b$ are incomparable);
\item[$b)$] Let $a\in C_i$, $b\in C_j$ and $i\neq j$. Then $a\leq b$ if and
only if $a=0$, whenever $i<j$;\\
$a\leq b$ if and
only if $x_i\leq b$, whenever $j<i$;
\item[$c)$] If $(0,1)$ is an adjunct pair $(${\it i.e.}, $x_i=1$ for
some $i\in \{1,2,\cdots, n\})$, then
$|V\left(G_{\{0\}}(L)\right)|=|L|-2$.
\end{enumerate}\end{lem}
Next result gives the structure of the zero-divisor graph of lower dismantlable lattices.
\begin{thm}\label{704}A simple undirected graph $G$ is complete $k$-partite if and only if $G=G_{\{0\}}(L)$ for some lower dismantlable lattice $L$ in which $1$ is the only adjunct element.
\end{thm}
\begin{proof}Let $G$ be a complete $k$-partite graph with $V_1,V_2,\cdots,V_k$ as partite
    sets. Without loss of generality, we can assume that $|V_i|\geq |V_{i+1}|$, for $i=1,2,\cdots,k-1$. Let $C_i$ be a chain
    such that $|C_i|=|V_i|$, for $i=2,3,\cdots,k$ and $|C_1|=|V_1|+2$. Then form a lattice
    $L=C_1]_0^1C_2]_0^1\cdots]_0^1C_k$. Clearly, $L$ is a lower dismantlable lattice with $(0,1)$ as the only adjunct
    pair such that $G_{\{0\}}(L)=G$.

    Conversely, suppose that $L$ is a lower dismantlable lattice with 1 as the only adjunct element. Hence $L=C_1]_0^1C_2]_0^1\cdots]_0^1C_k$. Then
    $G_{\{0\}}(L)\neq \emptyset$. In fact, by Lemma \ref{400}, two vertices are adjacent if
    and only if they belong to different chains in the adjunct
    representation of $L$. Therefore $G_{\{0\}}(L)$ is complete
    $k$-partite with $C_1\backslash \{0,1\}, C_2, \cdots,C_k$ as
    partite sets.
\end{proof}

 \begin{defn}\label{d2}Let $L$ be a finite lattice and $x\in L$. We say that $x$ is a {\it structurally deletable element}
 of $L$ if $x\in Irr(L\backslash \{0,1\})$ and $|E(CG(L))|=|E(CG(L\backslash \{x\}))|+1$. Delete the  structurally deletable
 element from $L$ and perform the operation of deletion till there does not remain any structurally deletable element.
 The resultant sublattice of $L$  is called the {\it basic block associated to $L$} and it is denoted by $B(L)$.
 \end{defn}

 \begin{ilu}In Figure \ref{f2}, a lower dismantlable lattice $L$ and its basic block $B(L)$ from $L$ are depicted. The procedure of obtaining $B(L)$ from $L$ is explained below.\\
 Since $C:a_1\prec a_2\prec a_3\prec a_4$ is a maximal chain of doubly irreducible elements with
  $0\prec a_1$, $a_4\prec x_1$ and $0,x_1\in Red(L)$. Then by Definition \ref{d2}, we remove the elements of $C$ except $a_1$. Repeat this procedure for the chains $a_9\prec a_{10}\prec a_{11}$, $a_6\prec a_7$,
  $x\prec a_8$. Also remove 1. Lastly, we get the basic block $B(L)$ associated to the lattice $L$.
 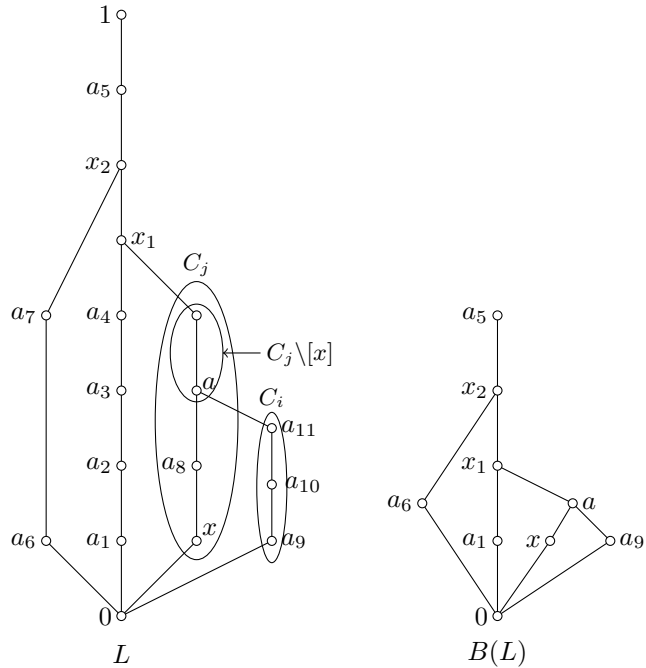
\begin{figure}[h]
 \begin{tikzpicture}
 \draw (0,0)--(0,8); \draw (0,0)--(-1,1)--(-1,4)--(0,6);
 \draw (0,5)--(1,4)--(1,1)--(0,0); \draw (1,3)--(2,2.5)--(2,1)--(0,0);

 \draw[fill=white](0,7) circle(.06);\draw[fill=white](0,8) circle(.06);
  \draw[fill=white](1,4) circle(.06); \draw[fill=white](1,3) circle(.06); \draw[fill=white](1,2) circle(.06);
   \draw[fill=white](1,1) circle(.06); \draw[fill=white](2,2.5) circle(.06); \draw[fill=white](2,1) circle(.06);
    \draw[fill=white](2,1.75) circle(.06);
 \draw[fill=white](0,0) circle(.06); \draw[fill=white](0,1)
 circle(.06); \draw[fill=white](0,2) circle(.06);
 \draw[fill=white](0,3) circle(.06); \draw[fill=white](0,4)
 circle(.06); \draw[fill=white](0,5) circle(.06);
 \draw[fill=white](0,6) circle(.06); \draw[fill=white](-1,1)
 circle(.06); \draw[fill=white](-1,4) circle(.06);

 \draw (1,3.5) ellipse (.35
 and .65); \draw[<-] (1.35,3.5)--(1.85,3.5);
 \draw (1,2.6) ellipse (.55
  and 1.85); \draw (2,1.71) ellipse (.2
   and 1);

  \node [left] at
  (0,0){$0$}; \node [left] at
  (0,1){$a_1$}; \node [left] at
  (0,2){$a_2$}; \node [left] at
  (0,3){$a_3$}; \node [left] at
  (0,4){$a_4$};
   \node [left] at
   (0,6){$x_2$}; \node [left] at
   (0,7){$a_5$};
  \node [right] at
  (0,5){$x_1$}; \node [left] at
  (0,8){$1$}; \node [left] at
  (-1,1){$a_6$}; \node [left] at
  (-1,4){$a_7$};

   \node [right] at
   (2,1){$a_9$};
    \node [left] at
    (1,2){$a_8$}; \node [right] at
    (2,2.5){$a_{11}$}; \node [right] at
    (2.05,1.75){$a_{10}$};

  \node [right] at
  (.95,3.1){$a$};
 \node [above] at
   (2.38,3.2){\small $C_j\backslash [x]$};
  \node [right] at
 (.95,1.15){$x$}; \node  at
 (0,-.5){$L$};\node [above] at
  (1,4.4){\small $C_j$}; \node  at
   (2,2.9){\small$C_i$};

 \draw (5,0)--(4,1.5)--(5,3)--(5,0); \draw (5,3)--(5,4);
 \draw (5,2)--(6,1.5)--(5.7,1)--(5,0); \draw (6,1.5)--(6.5,1)--(5,0);

 \draw[fill=white](5,0)
 circle(.06);\draw[fill=white](4,1.5) circle(.06);\draw[fill=white](5,3)
 circle(.06);\draw[fill=white](5,2)
 circle(.06);\draw[fill=white](5,1)
 circle(.06);\draw[fill=white](6,1.5)
 circle(.06);\draw[fill=white](5.7,1)
 circle(.06);\draw[fill=white](6.5,1)
 circle(.06);\draw[fill=white](5,4)
  circle(.06);\

  \node [left] at (5,0){$0$};
  \node [left] at (4,1.5){$a_6$};
  \node [left] at (5,1){$a_1$};
  \node [left] at (5,2){$x_1$};

  \node [left] at (5,3){$x_2$};
  \node [right] at (6,1.5){$a$};
  \node [left] at (5.7,1){$x$};
  \node [right] at (6.5,1){$a_9$};
 \node [left] at (5,4){$a_5$};

 \node  at
 (5,-.5){$B(L)$};
 \end{tikzpicture}
 \caption{A lower dismantlable lattice and its basic block}\label{f2}
 \end{figure}
\end{ilu}

\begin{obs} If $B$ is a basic block of a lattice $L$ then an element $b$ is structurally deletable if and only if
 $x\rightarrow b\rightarrow y$ is the only directed path from $x$ to $y$ in the cover graph $CG(L)$ of $L$, where $x\prec b \prec y$ in $L$.
\end{obs}
 A lattice $L$ with 0 is called {\it section semi-complemented} (in brief, SSC) if, for any $a,b\in L$ and $a\nleq b$
 there exists $c\in L$ such that $0 < c \leq  a$ and $b\wedge  c = 0$; see Janowitz \cite{jan} (see also Joshi \cite{19}).

 \begin{thm}
    \label{ssc}Let $L$ be a lower dismantlable lattice such that the greatest element $1$ of $L$ is join-reducible.
    Then the following statements are equivalent.
    \begin{enumerate}
        \item[$(a)$]The basic block of $L$ is $L$ itself.
        \item[$(b)$]$L$ is $SSC$.
        \item[$(c)$] Every equivalence class of the zero-divisor graph $G_{\{0\}}(L)$ is singleton.
    \end{enumerate}
 \end{thm}
 \begin{proof} $(a)\Rightarrow (b)$ Suppose that $L$ is a basic block of $L$ itself. On the contrary, suppose $L$ is not $SSC$.
    Hence there exists a pair of elements $b\nleq a$ with $a\wedge x\neq 0$ for every nonzero $x\leq b$. Then every nonzero
    $x\leq b$ is comparable with $a$.  We claim that $b$ is a join-irreducible element. Suppose $b$ is not join-irreducible, then $b=x\vee y$ for $x,y\prec b$. Clearly $a\nleq x,y$. Otherwise, since $L$ is lower dismantlable, $x|| y$. We have by Lemma \ref{400}, $x\wedge y=0$. Hence $a=0$, a contradiction to $a\wedge x\neq 0$ for any $x\leq b$.
    But then the only possibility is $x,y\leq a$ which yields $b=x\vee y\leq a$, again a contradiction. Thus $b$ is join-irreducible.  Since $L$ is a
    lower dismantlable lattice, we get $L= C_0]_{0}^{a_1} C_1]_{0}^{a_2}\cdots]_{0}^{a_{n}} C_n$, where each $C_i$ is a chain.
    By Remark \ref{r1}, every nonzero element of $L$ is meet-irreducible. In particular, $b$
    is meet irreducible, hence $b$ is doubly irreducible. Let $w_1$ and $w_2$ be the unique upper and lower
    covers of $b$ respectively. Then $w_1\rightarrow b\rightarrow w_2$ is the unique directed path from $w_1$ to $w_2$ in the cover graph of $L$, {\it i.e.}, $CG(L)$.
    Hence the element $b$ is structurally deletable and by deleting $b$, there is an edge joining $w_1$ to $w_2$, a contradiction
     to the fact
    that $L$ is a basic block of $L$ itself. Hence $L$ is $SSC$.\\
    $(b)\Rightarrow (a)$ Suppose $L$ is $SSC$. If $L$ is not a basic block of $L$ itself, then there exists a doubly irreducible element, say $b$,
    which is structurally deletable. We claim that $b$ is an atom. Suppose on the contrary that there exists nonzero element $x\prec b$. Then by the definition of $SSC$, there exists nonzero $c\leq b$
    such that $x\wedge c=0$. Since $x\prec b$ and $c\leq b$, we get
     $x\vee c=b$, a contradiction to the choice of $b$. Hence $b$ is an atom in $L$. Let $w$ be the unique upper cover of
    $b$. Then $0 \rightarrow b \rightarrow w$ is a directed path from $0$ to $w$ in $CG(L)$. If $w$ is join-reducible, then there is
    another path  $0\rightarrow c \rightarrow w$ in $CG(L) $, hence $b$ can not be structurally deletable, a contradiction.
    If $w$ is join-irreducible,
    then $0 \rightarrow b \rightarrow w$ is the only directed path from $0$ to $w$ in $CG(L)$. Thus  $b\leq w$  and there is no
    element $y\leq w$ such that $b\wedge y=0$, a contradiction, as $L$ is SSC. Hence $L$ is a basic block of $L$ itself.\\
    $(c)\Leftrightarrow (b)$ follows from Joshi et al. \cite[Lemma 9]{jwp1}.\end{proof}

Joshi et al. \cite{jwp1} proved that Isomorphism Problem is true for the class of $SSC$ meet
semi-lattices.
\begin{thm}[Joshi et al. \cite{jwp1}]\label{t2} Let $\mathscr{L}$ denotes the class of $SSC$ meet semi-lattices. Then  $L_1\cong L_2$ if and
only if $G_{\{0\}}(L_1)\cong G_{\{0\}}(L_2)$ for $L_1,L_2\in \mathscr{L}$.
\end{thm}
\begin{rem}In view of Theorem \ref{ssc} and Theorem \ref{t2}, it is clear that Isomorphism Problem is true for the class of lower dismantlable lattices which
are basic block of itself.

In the next section, we prove that Isomorphism Problem is true for  the larger class of lattices, namely, the class of lower dismantlable lattices.
         \end{rem}

\section{main results}
Let $T$ be a rooted tree with the root $R$ having at least
two branches. Let $G(T)$ be the {\it non-ancestor graph} of $T$ with vertex set $V(G(T))=T\backslash \{R\}$ and two vertices
are adjacent in $G(T)$ if and only if no one is an ancestor of the other. Denote the class of non-ancestor graphs of rooted trees by
$\mathcal{G_T}$.

In \cite{ap},  the following result is proved for the zero-divisor graphs of lower dismantlable lattices.

\begin{thm}[Patil et al. \cite{ap}]\label{rp} For a simple undirected graph $G$, the following statements are equivalent.
    \begin{enumerate}
        \item[$(a)$] $G\in \mathcal{G_T}$, the class of non-ancestor
        graphs of rooted trees. \item[$(b)$] $G=G_{\{0\}}(L)$ for
        some lower dismantlable lattice $L$ with the greatest element $1$
        as a join-reducible element. \item[$(c)$] $G$ is the
        incomparability graph of $(L\backslash \{0,1\},\leq) $ for some
        lower dismantlable lattice $L$ with the greatest element $1$ as a
        join-reducible element.
    \end{enumerate}
\end{thm}
\begin{rem}\label{r2} From Theorem \ref{rp}, it is clear that there is a one-to-one correspondence between the lower dismantlable lattices and rooted trees.
In fact, a lattice $L$ is a lower dismantlable lattice if and only if $L\backslash \{0\}$ is a rooted tree with the root 1. On the other hand, given a rooted tree $T$ with the root $R$, we join an element say 0 to all the pendent vertices of $T$ and get a cover graph of a lower dismantlable lattice $L$ in which $R$ is the greatest element and $0$ is the smallest element of $L$. We call $T$ as the {\it corresponding rooted tree} of $L$. Hence in view of Theorem \ref{rp}, it is clear that $G_{\{0\}}(L)=G(T)$ for a lower dismantlable lattice $L$ and its corresponding rooted tree $T$. Therefore the equivalence classes of $G(T)$ are same as the equivalence classes of $G_{\{0\}}(L)$.
\end{rem}
\noindent
{\bf Note:} In a lower dismantlable lattice which is not a chain, every adjunct element contains at least two atoms.

In the following construction, we give an algorithm to determine all equivalence classes of $G_{\{0\}}(L)$, where $L$ is a lower dismantlable lattice.
\begin{con}
Let $T$ be a rooted tree and $G(T)$ be the non-ancestor graph of $T$. A vertex $v$ of $T$ is a {\it node} if the total degree $deg(v)>2$ in $T$.
If $L$ is a lower dismantlable lattice with the corresponding rooted tree $T$, then $a(\neq 1)$ is an adjunct element in $L$ if and only
if $a$ is a node in $T$.

Let $v$ be a  node of $T$ such that no successor of $v$ is a node of $G(T)$. Then each branch with a successor
of $v$, {\it i.e.}, a directed path in $T$ of which every element is a successor of $v$,  is an equivalence class in $G(T)$ under the relation $\sim$ ({\it i.e.} having same neighbors). Delete all such branches and look at the resultant rooted tree $T'$. Repeat this process in $T'$ and so on, we get all the equivalence classes of $G(T)$. \end{con}
 We illustrate this procedure with an example.
\begin{exa}\label{ex2}In Figure \ref{f3}, a lower dismantlable lattice $L$, its corresponding rooted tree $T$ and its zero-divisor graph $G_{\{0\}}(L)$ are depicted. In the corresponding rooted tree $T$, $a_5$, $a_6$ are nodes having no successor as a node, whereas $a_8$ is a node having $a_5$ and $a_6$ as successor nodes. Hence delete vertices $a_1,a_2,a_3,a_4,a_7$. This gives the equivalence classes $\{a_1,a_7\},\{a_2\},\{a_3\}$, and $\{a_4\}$. Note that these equivalence classes do not contain an adjunct element. Now, the resultant rooted tree with the root 1 and $a_8$ is a node without successor node. Therefore $\{a_5\}$ and $\{a_6\}$ are the equivalence classes of $G(T)$ which contain an adjunct element of $L$. In the last stage we get $\{a_8\}$ as an equivalence class. Thus in this way, we get all the equivalence classes of $G(T)$.
 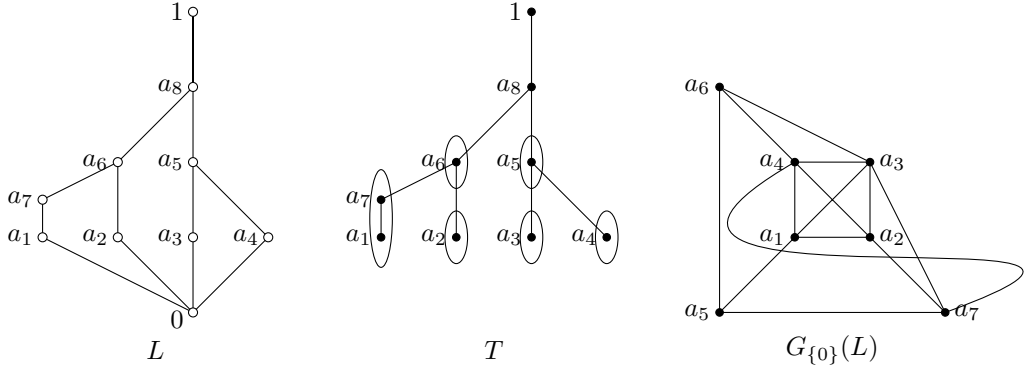
\begin{figure}[h]
    \begin{tikzpicture}
\draw (-2,1.5)--(-1,2)--(0,3)--(0,4)--(0,0);
\draw (-2,1.5)--(-2,1)--(0,0)--(-1,1)--(-1,2);
\draw (0,2)--(1,1)--(0,0);
\draw[fill=white](0,0) circle(.06);    \draw[fill=white](0,1) circle(.06);   \draw[fill=white](0,2) circle(.06);
\draw[fill=white](0,3) circle(.06);
\draw[fill=white](0,4) circle(.06);      \draw[fill=white](-2,1.5) circle(.06);
\draw[fill=white](-1,2) circle(.06);    \draw[fill=white](-2,1) circle(.06);   \draw[fill=white](-1,1) circle(.06);
\draw[fill=white](1,1) circle(.06);
 \node [left] at (0,-.1){$0$}; \node [left] at (0,3){$a_8$};  \node [left] at (0,1){$a_3$};   \node [left] at (0,2){$a_5$};   \node [left] at (0,4){$1$};
  \node [left] at (-2,1.5){$a_7$};   \node [left] at (-1,2){$a_6$};   \node [left] at (-2,1){$a_1$};
   \node [left] at (-1,1){$a_2$};    \node [left] at (1,1){$a_4$};
\node [] at (-.5,-.5){$L$};

 \draw (2.5,1.25) ellipse (.15
 and .65);
\draw (3.5,1) ellipse (.15
and .35); \draw (3.5,2) ellipse (.15
and .35); \draw (4.5,1) ellipse (.15
and .35);
\draw (5.5,1) ellipse (.15
and .35); \draw (4.5,2) ellipse (.15
and .35);

\draw (2.5,1.5)--(3.5,2)--(4.5,3)--(4.5,4)--(4.5,1);
\draw (2.5,1.5)--(2.5,1); \draw (3.5,1)--(3.5,2);
\draw (4.5,2)--(5.5,1);
 \draw[fill=black](4.5,1) circle(.05);   \draw[fill=black](4.5,2) circle(.05);  \draw[fill=black](4.5,4) circle(.05);
\draw[fill=black](4.5,3) circle(.05);      \draw[fill=black](2.5,1.5) circle(.05);
\draw[fill=black](3.5,2) circle(.05);    \draw[fill=black](2.5,1) circle(.05);   \draw[fill=black](3.5,1) circle(.05);
\draw[fill=black](5.5,1) circle(.05);
  \node [left] at (4.5,1){$a_3$};   \node [left] at (4.5,2){$a_5$};  \node [left] at(4.5,3) {$a_8$}; \node [left] at (4.5,4){$1$};
  \node [left] at (2.5,1.5){$a_7$};   \node [left] at (3.5,2){$a_6$};   \node [left] at (2.5,1){$a_1$};
\node [left] at (3.5,1){$a_2$};    \node [left] at (5.5,1){$a_4$};
\node [] at (4,-.5){$T$};

\draw (8,1)--(8,2)--(9,2)--(9,1)--(8,1)--(9,2)--(10,0)--(9,1)--(8,2)--(7,3)--(7,0)--(10,0); \draw (7,3)--(9,2);
\draw (7,0)--(8,1);

\draw[fill=black](7,0) circle(.05);
\draw[fill=black](7,3) circle(.05);
\draw[fill=black](8,1) circle(.05);
\draw[fill=black](8,2) circle(.05);\draw[fill=black](9,1) circle(.05);
\draw[fill=black](9,2) circle(.05);
\draw[fill=black](10,0) circle(.05);
%\path (10,0) edge [out=55,in=-30] (8,2);

\draw (10,0) to[out=20, in=210, looseness=4] (8,2);

\node [left] at (8,1){$a_1$};  \node [right] at (9,2){$a_3$};  \node [left] at (7,0){$a_5$};  \node [left] at (7,3){$a_6$};
\node [right] at (9,1){$a_2$};  \node [left] at (8,2){$a_4$};  \node [right] at (10,0){$a_7$};
\node [] at (8.5,-.5){$G_{\{0\}}(L)$};
\end{tikzpicture}
\caption{A lower dismantlable lattice with its corresponding rooted tree and zero-divisor graph}\label{f3}
\end{figure}
\end{exa}

The following lemma gives more about equivalence classes of $G_{\{0\}}(L)$.
\begin{lem}\label{l14}Let $L$ be a lower dismantlable lattice and $G_{\{0\}}(L)$ be its zero-divisor graph. Then the  following statements are true.
\begin{enumerate}\item[$(a)$] Every equivalence class of $G_{\{0\}}(L)$ forms a chain in $L$. Moreover, every equivalence class contains
at most one adjunct element.
        \item[$(b)$] If $a$ is an adjunct element in the equivalence class
        $[x]$ of $G_{\{0\}}(L)$, then $a\leq x$. \item[$(c)$] An equivalence class $[x]$ of $G_{\{0\}}(L)$ contains an adjunct element of $L$ if and only if $[x]$ of $G(T)(=G_{\{0\}}(L))$
does not contain a pendent vertex of the corresponding rooted tree $T$(mentioned in Remark \ref{r2}).
\item[$(d)$] The branches that we get by the procedure explained in Example \ref{ex2} are precisely the equivalence classes of $G_{\{0\}}(L)$.\end{enumerate} \end{lem}
\begin{proof}$(a)$ Let $[x]=\{y\in V(G_{\{0\}}(L)) ~|~ N(x)=N(y)\}$. Let $y_1,y_2\in [x]$. If $y_1|| y_2$, then
 $y_1\wedge y_2=0$, by Lemma \ref{400}. Hence $y_1\in N(y_2)=N(y_1)$, a contradiction. Hence $y_1$ and $y_2$ are comparable.
 Therefore $[x]$ is a chain. Next, suppose $y,z\in [x]$ be two adjunct elements. Without loss of generality,
 let $y< z$. Since $z$ is an   adjunct element, there exists a nonzero $w\leq z$ such that $y\wedge w=0$.
 Hence $w\in N(y)=N(z)$ which yields $w=w\wedge z=0$, a contradiction. Therefore  every equivalence
 class contains at the most one  adjunct element.\\
    $(b)$ As $a\sim x$, we get $a$ and $x$ are
    comparable. If $x<a$, then there exists a nonzero $b<a$ such that
    $x\wedge b=0$ (as $a$ is an adjunct element). Hence $b\in
    N(x)=N(a)$ which gives $b=a\wedge b=0$, a contradiction. Therefore
    $a\leq x$.\\
    $(c)$ It is easy to observe that  an element $y$ is an atom in $L$ if and only if $y$ is a pendent vertex
    of the corresponding rooted tree $T$ of $L$. Suppose that $a$ be an adjunct element in $[x]$. Hence $a\leq x$. On the
    contrary, assume that $p\in [x]$ be an atom. Then $p\leq a$. Since $a$ is an adjunct element, there exists
    an atom $q(\neq p)\leq a$. Hence $p\wedge q=0$, which gives $q\in N(p)=N(a)$. This yields $q=q\wedge a=0$,
    a contradiction. Therefore $[x]$ contains no atom. Conversely, suppose that $[x]$ of $G(T)$ contains no atom of $L$. %On the contrary, assume that $[x]$ contains no adjunct element.
     Let $p$ be an atom such that $p\leq x$. This gives $N(x)\subsetneqq N(p)$. Since $[x]$ contains no atom,
     there exists $y\in N(p)$ such that $y\wedge x\neq 0$. Then $x$ and $y$ are comparable, by Lemma \ref{400}. Since $p\leq x$,
     the case $x\leq y$ is impossible. Hence $y<x$. Then $p\vee y\leq x$. If $p\vee y\in [x]$ then we are
     through, as $p\vee y$ is an adjunct element. If not, {\it i.e.}, $p\vee y\notin [x]$ then again $N(x)\subsetneqq N(p\vee y)$.
     So there exists $c_1\in N(p\vee
     y)$ such that $c_1\notin N(x)$. This gives $x$ and $c_1$ are
     comparable, by Lemma \ref{400} $(a)$. Using the above arguments, we get $c_1\leq x$ and
     $N(x)\subseteq N(p\vee y\vee c_1)$. Continuing in this way we get
     an element say $c_n$ (as $L$ is finite) such that $N(x)=N(p\vee
     y\bigvee_{i=1}^n c_i)$. Then $p\vee y\bigvee_{i=1}^n c_i$ is
     an adjunct element such that $y\vee z\bigvee_{i=1}^n c_i
     \in [x]$.\\
     $(d)$ Follows from the fact that the equivalence classes of $G(T)$ are same as the equivalence classes of $G_{\{0\}}(L)$.
\end{proof}
\begin{cor}\label{701} An equivalence class
     $[x]$ of $G_{\{0\}}(L)$ for a lower dismantlable lattice $L$ contains an adjunct element if and only if there
     is a pair of vertices $y,z\in V(G_{\{0\}}(L))$ such that $y$ is
     adjacent to $z$ and $x$ is not adjacent to any of $y$ and $z$.
\end{cor}
\begin{proof}Suppose $a$ is an adjunct element in $[x]$. Hence $a\leq x$. Since $a$ is an adjunct element, there exist two atoms $p_1,p_2\leq a$. Then $p_1$ and $p_2$ are the required elements. Conversely, suppose that there
    is a pair of vertices $y,z\in V(G_{\{0\}}(L))$ such that $y$ is
    adjacent to $z$ and $x$ is not adjacent to any of $y$ and $z$. Let $p,q$ be atoms such that $p\in [x]$ and $q\leq y$.
    Since $x$ and $y$ are non-adjacent, they are comparable in $L$. The case $x\leq y$ is impossible, since $y$ and $z$ are adjacent but  $x$ and $z$
    are non-adjacent. Hence $y<x$, which further gives $q\leq x$. But $p\wedge q=0$ and $p\in [x]$ gives $q\in N(p)=N(x)$
    which yields $q\wedge x=0$, a contradiction. Hence $[x]$ does not contain an atom. By Lemma \ref{l14}, $[x]$ does not contain an adjunct element.
    \end{proof}
We need the following concept of ordinal sum of two posets.
\begin{defn}Let $P$ and $Q$ be disjoint posets. Let $P\cup Q$ be the union with the inherited order on $P$ and $Q$ such that $p<q$ for all $p\in P$ and $q\in Q$.  Then it forms a poset called the {\it ordinal sum} of $P$ and $Q$ denoted by $P\oplus Q$.\end{defn}

\begin{lem}\label{l16}If an equivalence class $[x]$ does not contain an  adjunct element then the set
 $A_x= \big\{b\in G_{\{0\}}(L)~|~ b\textnormal{\textit{ is an adjunct element of $L$ not adjacent to }}x\big\}$ is
either empty or forms a chain in $L$. In this case, $L=L_1]_0^a C$, where $L_1$ a is lower dismantlable lattice
and $C$ is a chain. Moreover, $a=1$ if and only if  $A_x=\emptyset$, and if $a\neq 1$, then $a$ is the smallest element of $A_x$.
\end{lem}
\begin{proof} Suppose $[x]$ does not contain any  adjunct element. Let
$L=C_0]_0^{x_1}C_1]_0^{x_2}C_2]_0^{x_3}\cdots]_0^{x_n}C_n$ and $x\in C_j$. Consider the set
  $A_x= \big\{b\in L~|~ b\textnormal{ is an adjunct element not adjacent to }x\big\}$ 

 If $A_x=\emptyset$, then
$C_j$ must be joined at $(0,1)$ and all the elements of $C_j$ have same neighbors, {\it i.e.}, $[x]=C_j$.
Then $L=L_1]_0^1C_j$, where
$L_1=C_0]_0^{x_1}C_1]_0^{x_2}\cdots]_0^{x_{j-1}}C_{j-1}]_0^{x_{j+1}}C_{j+1}]_0^{x_{j+2}}\cdots]_0^{x_n}C_n$.

Let $A_x\neq \emptyset$.  Suppose $b\in A_x$. Then $x$ is not adjacent to $b$. Hence $x$ and $b$ are comparable. Suppose $b\leq x$. Since $b$ is an adjunct element, there exists a pair
$y,z\leq b$ such that $y\wedge z=0$ and $x$ is not adjacent to any
of $y$ and $z$. By Corollary \ref{701}, $[x]$
contains an  adjunct element, a contradiction. Therefore
$x\leq b$.

Now, we prove that the elements of $A_x$ forms a chain. Let $b_1, b_2\in A_x$ such that $b_1||b_2$. Then
$b_1\wedge b_2=0$. As $b_1, b_2\in A$, we have $x\leq b_1$ and
$x\leq b_2$. Therefore $x\wedge b_1=0$, \textit{i.e.,} $x$ and
$b_1$ are adjacent,
 a contradiction to the choice of $b_1$. Therefore the elements of $A_x$ forms a chain. 
 
Let $a$ be the smallest element of $A_x$.
We claim that $[x]\subseteq C_j$. Let $y\sim x$, {\it i.e.},
$y\wedge z=0$ if and only if $x\wedge z=0$. As $x\wedge y\neq 0$,
they are comparable. If $y\leq x$, then it gives $y\in C_j$, otherwise
$[x]$ contains an adjunct element, a contradiction. Suppose
$x\leq y$ and $y\in C_i$ for $i\neq j$. By Lemma \ref{400} ($b$), we get $y\geq x_j$, where $x_j$ is an adjunct element. Therefore
there exist two elements
$y_1, y_2\leq x_j$ such that $y_1\wedge y_2=0$. As $x_j\leq y$, we
have $y_1$, $y_2$ are adjacent and $y$ is not adjacent to any of
them. By Corollary \ref{701}, $[y]=[x]$ (as $y\sim x$) contains
an adjunct element, a contradiction to the assumption. Hence
$y\in C_j$, {\it i.e.,} $[x]\subseteq C_j$.

Next, we claim that if $z\in C_j$ such that $z<a$, then $z\sim x$. Let $z\in C_j$ such that $z<a$.
Since $x,z\in C_j$, they are comparable.

If $x\leq z$, then $N(z)\subseteq N(x)$. If there exists $y\in N(x)$
but $y\notin N(z)$, then $y$ and $z$ are comparable. If  $z\leq y$, then $z\wedge x=0$, a contradiction.
 Hence $y\leq z$. Thus $x,y\leq z$ such that $x\wedge y=0$. Hence by Lemma \ref{l14} ($b$), there exists an adjunct element, say  $c\in [z]$,
  such that $x\leq c\leq z$, a contradiction to the minimality of $a$. Hence $x\leq z$ gives $N(x)=N(z)$.
  
  Suppose that $z\leq x$. Then $N(x)\subseteq N(z)$. If $y\in N(z)$ such that $y\notin N(x)$, then $y$ and $x$ comparable.
   If $x\leq y$, then $x\wedge z=0$, a contradiction. Hence $y\leq x$. This together with $z\leq x$, $y\in N(z)$ and by Corollary \ref{701}, we have  $[x]$ contains
   an adjunct element,  again a contradiction. Hence $N(z)=N(x)$, {\it i.e.}, $x\sim z$.

Further, suppose that an adjunct element $x_i$ is non-adjacent to $x$. Hence $x_i$ and $x$ are comparable.
If $x_i\leq x$, then by Corollary \ref{701}, $[x]$ contains an adjunct element, a contradiction. Hence $x< x_i$.
In particular, we have $x<a$.

 Now, we prove that $L=L_1]_0^aC$, where $L_1$ is a lower dismantlable lattice and $C$ is a chain. First, assume
  that $x\notin C_0$, {\it i.e.,} $x\in C_j$ for $j\neq
0$. Since $x<a$, we have the following two
cases.\\
\textbf{Case $(1)$:} If $a\notin C_j$. Since $x<a$ and $a\notin C_j$, by Lemma \ref{400}$(b)$, we get $x_j\leq a$.
By the minimality of $a$, we have $a=x_j$ and no element of $C_j$ is an
adjunct element. In this case
$[x]=C_j$, as all elements of $C_j$ have same neighbors. Define  $L'=L_1]_0^aC_j$ with the induced partial order of $L$,
 where
$L_1=C_0]_0^{x_1}C_1]_0^{x_2}C_2]_0^{x_3}\cdots]_0^{x_{j-1}}C_{j-1}]_0^{x_{j+1}}C_{j+1}]_0^{x_{j+2}}\cdots]_0^{x_n}C_n$.

 We claim that $L'=L$. Clearly, $|L|=|L'|$(as we are playing with the same elements). Let $y_1\leq y_2$ in $L$.
 If $y_1,y_2\notin C_j$ or $y_1,y_2\in C_j$, then $y_1\leq y_2$ in $L'$. Note that $y_2\in C_j$ implies $y_1\in C_j$,
  since no element of $C_j=[x]$ is an
 adjunct element. Suppose $y_1\in C_j$ and $y_2\in C_i$ for $i\neq j$ in $L$. Since $y_1\leq y_2$ and $y_1\notin C_i$,
  we get $x_j\leq y_2$ in $L$. Hence $y_1\leq x_j\leq y_2$ in $L'$. The converse follows similarly.
 Therefore $L=L'=L_1]_0^aC_j$.\\
\textbf{Case $(2)$:} Let $a\in C_j$. Then $a$ and $x$ are on the same chain $C_j$. Since $a$ is an adjunct element;
we have $a=x_i$, for some $i$. Hence $C_i$ is a chain  joined at $a$.

 Define
$L_1=C_0]_0^{x_1}C_1]_0^{x_2}\cdots]_0^{x_{j-1}}C_{j-1}]_0^{x_j}C_j']_0^{x_{j+1}}C_{j+1}]_0^{x_{j+2}}\cdots]_0^{x_{i-1}}C_{i-1}]_0^{x_{i+1}}C_{i+1}\cdots]_0^{x_n}C_n$,
where $C_j'=C_i\oplus(C_j\backslash [x])$ (see Figure \ref{f2}). Since $a\in C_j$ and $a\notin [x]$, $C_j\backslash [x]\neq \emptyset$.
Let
$L'=L_1]_0^{a=x_i}C$, where $C=[x]$ is a chain(by Corollary \ref{701}) under the induced partial order of $L$. We claim that $L=L'$.
 Observe
that $|L|=|L'|$ (as we are playing with the same elements).

 Suppose
$y_1\leq y_2$ in $L$. If they belong to the same chain, say $C_k$
($k\neq j$) in $L$, then $y_1\leq y_2$ in $L'$ also.

 Suppose
$y_1,y_2\in C_j$. If $y_1,y_2\sim x$, then $y_1,y_2\in C$ in $L'$
implies $y_1\leq y_2$ in $L'$. If $y_1,y_2\nsim x$, then
$y_1,y_2\in C_j'$ in $L_1$ implies $y_1\leq y_2$ in $L'$. Suppose
$y_1\sim x$ and $y_2\nsim x$, then $y_1\in C$ and $y_2\in C_j'$ in
$L'$. As $y_1\sim x$, $y_2\in C_j'$ and $a$ is an adjunct element, we have $y_1<a$ and
$y_2\geq a$, which gives $y_1\leq y_2$ in $L'$.

Now, suppose $y_1$ and $y_2$ are on different chains, say $y_1\in
C_p$ and $y_2\in C_k$ with $p\neq k$. As $y_1\leq y_2$ in $L$, we get
$y_2\geq x_p$ (as $C_p$ is joined $x_p$) hence $k<p$. Also
$y_1\leq x_p$, hence $y_1\leq x_p\leq y_2$ in $L'$.

Next, suppose  $y_1\leq
y_2$ in $L'$. If they belong to the same chain except $C_j'$, then
$y_1\leq y_2$ in $L$ also. Suppose $y_1, y_2\in C_j'$. If $y_1,
y_2\in C_j\backslash [x]$ or $y_1,y_2\in C_i$, then $y_1\leq y_2$
in $L$. If $y_2\in C_i$, then $y_1\in C_i$, as $y_1\leq y_2$. Let $y_2\in C_j\backslash [x]$ and $y_1\in C_i$. Then $y_1\leq
a\leq y_2$ in $L$. Suppose $y_1$ and $y_2$ are on different chains, say $y_1\in C_p$ and $y_2\in C_k$ with
$p\neq k$ in $L'$. As $y_1\leq y_2$, we have $k<p$ with $y_2\geq
x_p$ and $x_p\geq y_1$  {\it i.e.,} $y_1\leq x_p\leq y_2$ in $L$.

Therefore $L=L'$.

If $x\in C_0$, it gives $a\in C_0$. Then $x$ and $a$ are on the same chain $C_0$, hence by Case (2) above, we
have $L=L_1]_0^aC$, where $C=[x]$.
\end{proof}

 \begin{thm}
    \label{703} Let $L_1$, $L_2$ be lower dismantlable lattices with $1$ as an adjunct element such that
    $G_{\{0\}}(L_1)\cong G_{\{0\}}(L_2)$. Then there is a graph isomorphism
    $\phi:G_{\{0\}}(L_1)\rightarrow G_{\{0\}}(L_2)$  such that $a$ is an
    adjunct element of $L_1$ if and only if $\phi (a)$ is an adjunct element of $L_2$.\end{thm}
 \begin{proof} Let  $f: G_{\{0\}}(L_1) \rightarrow G_{\{0\}}(L_2)$ be a graph isomorphism, where $L_1$ and $L_2$
 be lower dismantlable
    lattices. Since $1$ is an adjunct element, we have $V(G_{\{0\}}(L_i))=L_i\backslash \{0, 1\}$, for $i=1,2$.
    First, we prove that $f([x])=[f(x)]$ for all $x\in V(G_{\{0\}}(L_1))$, where  $f([x])=\{f(a)~|~ a\in [x]\}$
    and $[x]=\{y\in L_1| N(x)=N(y)\}$.

    Let $t\in f([x])$. Then $t=f(a)$ for some $a\in [x]$. Therefore $N(a)=N(x)$ implies $N(f(a))=N(f(x))$ by the graph isomorphism.
     Hence $t\in [f(x)]$. Thus $f([x])\subseteq[f(x)]$.
    Let $f(a)\in [f(x)]$. Hence $N(f(a))=N(f(x))$ implies $N(a)=N(x)$ by the graph isomorphism. Therefore $a\in [x]$
    which yields $f(a)\in f([x])$.
    Hence $[f(x)]\subseteq f([x])$. Thus $[f(x)]= f([x])$.

    It is clear that $y\sim x$ if and only if $f(y)\sim f(x)$. This gives $|[x]|= |[f(x)]|$.

    Now, we claim that $[x]$ contains an adjunct element if and only if $[f(x)]$ contains an adjunct element.
    Suppose $[x]$ contains an adjunct element of $L_1$.
    Then by Corollary \ref{701}, there exist elements $y,z\in L_1$ such that $x$ is not adjacent to any of $y$ and $z$
    with $y\wedge z=0$. Hence $f(x)$ is not adjacent to any of $f(y)$ and $f(z)$ with $f(y)\wedge f(z)=0$. Therefore $[f(x)]$
    contains an adjunct element of $L_2$. Converse follows on the similar lines.\\
    Define $A_f=\{x~|x \textnormal { is an adjunct element of } L_1 \textnormal { and } f(x) \textnormal
    { is not an adjunct element of } L_2\}$. We prove the result by the induction on $|A_f|$.
    
    If $A_f=\emptyset $,
    then $f$ is the required isomorphism.
    
    Let $A_f\neq \emptyset$ and assume the result is true for all  lower dismantlable lattices with $|A_f|<k$. Now, suppose
    $|A_f|=k$.
    Let $A_f=\{x_1,x_2,\cdots,x_k\}$. Since $x_i$ is an adjunct element, $[x_i]$ contains an adjunct element.
    Hence $[f(x_i)]$ contains an adjunct element, but $f(x_i)$ is not an adjunct element, as $x_i\in A_f$.
    So $|[x_i]| =|[f(x_i)]|>1$ for $i=1,2,\cdots,k$. Let $[x_1]=\{x_{11},x_{12},\cdots,x_{1m}\}$ with $x_1=x_{11}$.
    Then $[f(x_1)]=\{f(x_{11}),f(x_{12}),\cdots,f(x_{1m})\}$. Without loss of generality, suppose $f(x_{12})$ be an adjunct element of
     $L_2$.
     
     Next define
    $\phi_1:G_{\{0\}}(L_1)\rightarrow G_{\{0\}}(L_2)$ by $\phi_1(y)=f(y)$, for all $y\in L_1\backslash \{x_{11},x_{12}\}$,
    $\phi_1(x_1)=f(x_{12})$ and $\phi_1(x_{12})=f(x_1)$.
    Then $\phi_1$ is bijective. Let $x$ and $y$ be adjacent in $G_{\{0\}}(L_1)$. Then $|\{x,y\}\cap \{x_{11},x_{12}\}|< 2$.
    Note that $f(x_1)=f(x_{11})$ is not an adjunct element, as $x_{11}$ and $x_{12}$ are non-adjacent.

    If $\{x,y\}\cap \{x_{11},x_{12}\}=\emptyset$, then $x,y\in L_1\backslash \{x_{11},x_{12}\}$. Hence $\phi_1(x)$ and $\phi_1(y)$
    are adjacent in $G_{\{0\}}(L_2)$ by the definition of $\phi_1$.  Suppose $|\{x,y\}\cap \{x_{11},x_{12}\}|=1$.
    Without loss of generality, suppose that $x=x_{11}$.
    As $x_{11}\sim x_{12}$, we have $x_{12}$ and $y$ are adjacent in $G_{\{0\}}(L_1)$. Therefore $f(x_{12})$
    and $f(y)$, {\it i.e.,} $\phi_1(x)$
    and $\phi_1(y)$ are adjacent in $G_{\{0\}}(L_2)$. Also using the above arguments for $\phi_1^{-1}$, we
    have $x$ and $y$ are adjacent in
    $G_{\{0\}}(L_1)$ whenever $\phi_1(x)$ and $\phi_1(y)$ are adjacent in $G_{\{0\}}(L_2)$. Hence $\phi_1$ is a graph isomorphism.

    Let $a\in A_{\phi_1}$. Then $a$ is an adjunct element of $L_1$ and $\phi_1(a)$ is not an adjunct element of $L_2$.
     We claim that $a\notin \{x_{11}, x_{12}\}$. On the contrary, assume that $a\in \{x_{11}, x_{12}\}$. If $a=x_{11}=x_1$,
      then $\phi_1(a)=\phi_1(x_{11})=f(x_{12})$ is an adjunct element of $L_2$, a contradiction. Also  $a=x_{12}$ impossible
      because $a$ is an adjunct element of $L_1$ but $x_{12}$ is not, since each equivalence class contains at the most one
       adjunct element and $x_1=x_{11}$ is an adjunct element in $[x_1]$ with $x_{12}\in [x_1]$.
       Therefore $a\notin \{x_{11}, x_{12}\}$. This gives $f(a)=\phi_1(a)$ is not an adjunct element of $L_2$,
       hence $a\in A_f$. Thus $A_{\phi_1}\subseteq A_f$. But $x_1\in A_f$ with $x_1\notin A_{\phi_1}$.
       This gives $A_{\phi_1}\subsetneqq A_f$. Therefore $|A_{\phi_1}|< |A_f|$.
    Hence by induction, there is an isomorphism $\phi:G_{\{0\}}(L_1)\rightarrow G_{\{0\}}(L_2)$ such that $a$ is
    an adjunct element of $L_1$  if and only if $\phi (a)$ is an adjunct element
    of $L_2$ . Hence the result.
 \end{proof}

\begin{thm}
\label{706}Let $L_1$ and $L_2$ be lower dismantlable lattices with $1$ as adjunct element. If
$\phi:G_{\{0\}}(L_1)\rightarrow G_{\{0\}}(L_2)$ is a graph  isomorphism
such that $a$ is an adjunct element in $L_1$ if and only if $\phi(a)$ is an adjunct element in $L_2$. Then there exists an
isomorphism $\psi:L_1\rightarrow L_2$ such that $\psi_{|_X}\equiv
\phi_{|_X}$, where $X$ is the set of adjunct elements of $L_1$
different from $1$. Moreover, for any equivalence class $[x]$ in
$G_{\{0\}}(L_1)$, we have $\psi([x])=\phi([x])$.
\end{thm}
\begin{proof}We use the induction on the number of vertices. By Theorem \ref{403}, $G_{\{0\}}(L_i)$ are connected, for $i=1,2$. We know that if there are only two vertices
then the graphs $G_{\{0\}}(L_i)$, $i=1,2$ are isomorphic to $K_2$
 and therefore the lattices are isomorphic to the power set of two elements. In this case, $1$ is the
 only adjunct element, hence $X=\emptyset$.

Now, suppose $G_{\{0\}}(L_1) \cong G_{\{0\}}(L_2)$ with
$|G_{\{0\}}(L_i)|>2$, for $i=1,2$. Suppose $L_1$ and $L_2$ satisfy the
hypothesis.  Select $x\in L_1$ such that the equivalence class
$[x]$ in $G_{\{0\}}(L_1)$ does not contain any adjunct element of
$L_1$. Note that, by Lemma \ref{l14}, such an equivalence class
$[x]$ exists. Then by the hypothesis $[\phi(x)]$ also does not contain any adjunct
element of $L_2$. By Lemma \ref{l16}, we can write
$L_1=L'_1]_0 ^a C$ and $L_2=L'_2]_0 ^{a'} C'$, where $C=[x]$ and
$C'=[\phi(x)]$ and either $a,a'$ both are corresponding $1$'s of
$L_1$ and $L_2$ respectively or the smallest elements of
corresponding set of all adjunct elements which are comparable
with $x$, $\phi(x)$ respectively.

 We claim that
$\phi(a)=a'$. As $a$ and $x$ are non-adjacent in $G_{\{0\}}(L_1)$,
we have $\phi(a)$ and $\phi(x)$ are non-adjacent in
$G_{\{0\}}(L_2)$. Hence they are comparable in $L_2$ with
$\phi(a)$ as an adjunct element in $L_2$. But $a'$ is the smallest
adjunct element which is not adjacent to $\phi(x)$, hence $a'\leq
\phi(a)$. Let $b\in L_1$ such that $\phi(b)=a'$. Then $b$ is an
adjunct element in $L_1$ since $\phi(a)$ is an adjunct element and if $a'<\phi(a)$, then there exists an
element $\phi(c)<\phi(a)$ such that $a'\wedge \phi(c)=0$, {\it
i.e.,} $\phi(b)\wedge \phi(c)=0$. By the graph isomorphism, $b\wedge c=0$ in $L_1$.
 Also, both $b$ and $c$ are comparable with $a$. The only possibility is $b,c< a$. If $b$ is adjacent to $x$ in
 $G_{\{0\}}(L_1)$, then $a'=\phi(b)$
 is adjacent to $\phi(x)$ in $G_{\{0\}}(L_2)$, a contradiction to the fact that $\phi(x)$ and $\phi(b)$ are comparable,
  as $\phi(x)$ is on the chain $C'$ and $C'$ is joined at $a'=\phi(b)$. Hence $b$ is not adjacent to $x$ in $G_{\{0\}}(L_1)$
 and $b$ is an adjunct element
 with $b<a$, a contradiction to the smallestness of $a$. Hence $a'=\phi(a)$.

  Among all such equivalence classes select
 one $[x]$ for which corresponding
 element $a$ is minimal among such adjunct elements (minimal in the sense that, if $b$ another adjunct element, then
 either $a||b$ or $a\leq b$). Then
 $[\phi(x)]$ is an equivalence class in $G_{\{0\}}(L_2)$ such that $a'$ is minimal among such adjunct elements of $L_2$.
 Next, we consider the following cases for $a$.\\
\textbf{Case $(1)$: } Suppose $a=1$. In this case there is no adjunct pair other than $(0,1)$, hence $X=\emptyset$ and
the result follows by Theorem \ref{704}.\\
\textbf{Case $(2)$: } $a\neq 1$. Then we have $a'\neq 1$, otherwise $1$ is the only adjunct element
in $L_2$(by  the minimality of $a'$). By Theorem \ref{704}, we get $G_{\{0\}}(L_2)$ is complete bipartite,
so is $G_{\{0\}}(L_1)$ and again by Theorem \ref{704}, we get $1$ is the only adjunct element in $L_1$,
 a contradiction to the fact that $a\neq 1$. Thus we must have $a'=1$. Consequently, in the lower
 dismantlable lattices $L'_1$ and $L_2'$, the corresponding greatest elements $1$ are join-reducible.
 Note that $|[x]|=|[\phi(x)]|$, {\it i.e.}, $|C|=|C'|$ follows as in the proof of Theorem \ref{703}. Also
$G_{\{0\}}(L'_1)=G_{\{0\}}(L_1) \backslash [x]\cong
G_{\{0\}}(L_2)\backslash [\phi(x)]=G_{\{0\}}(L'_2)$, under the map
$\phi_{|_{L_1'}}$. Hence by the induction hypothesis,  there exists an
isomorphism $\psi:L'_1\rightarrow  L'_2$ such that $\psi_{|_{X_1}}\equiv
\phi_{|_{X_1}}$, where $X_1$ is the set of adjunct elements of $L_1'$
different from $1$ and for any equivalence class $[y]$ in
$G_{\{0\}}(L_1'), ~\psi([y])=\phi([y])$.

Suppose $a$ is an adjunct
element in $L_1'$. Then $\psi(a)$ is an adjunct element in $L_2'$
with $\psi(a)=\phi(a)=a'$. Since $C$ and $C'$ are chains of same
length both without containing an adjunct element, hence we can extend
$\psi$ to $L_1=L_1']_0^aC$ which gives $L_1\cong L_2$ and
$\psi_{|_X}\equiv \phi_{|_X}$, where $X$ is the set of adjunct
elements of $L_1$ different from $1$, and in this case $X=X_1$. Also, for any equivalence class $[x]$
in $G_{\{0\}}(L_1)$, we have $\psi([x])=\phi([x])$.

 Now, suppose $a$ is not an adjunct element in $L_1'$.
We claim that $[a]$ in $G_{\{0\}}(L_1')$ does not contain any
adjunct element. If $b\in [a]$ be an adjunct element in
$G_{\{0\}}(L_1')$, then by Lemma \ref{l14} ($a$), we get $b<a$ in $L_1'$ and hence in $L_1$. Let $C_a$ be
a chain in the adjunct representation of $L_1$ such that $a\in C_a$
and $y$ is an atom of $L_1$ on the chain $C_a$. As $a,b\in L_1'$ and there
is only one chain  $C$ joined at $a$ in $L_1$, we have
$b\in C_a$. Also $[y]$ in $G_{\{0\}}(L_1)$ does not contain any
adjunct element of $L_1$ and $b$ is an adjunct element in $L_1$ which is
comparable with $y$ such that $b<a$, a contradiction to the choice
of $a$. Hence $[a]$ in $G_{\{0\}}(L_1')$ does not contain any
adjunct element  of $L_1'$.

Similarly, $[a']$ in $G_{\{0\}}(L_2')$ does not
contain any adjunct element of $L_2'$.

Next, for the equivalence class
$[a]$ in $G_{\{0\}}(L_1')$, we have
$\psi([a])=\phi_{|_X}([a])=\{\phi(y)|y\in [a]\}$. Hence
$a'=\phi(a)\in \psi([a])$. Let $\psi(b)=a'$, for $b\in [a]$. Then
$b\sim a$, hence $a$ and $b$ are comparable in $L_1'$. This gives
$\psi(a)$ and $\psi(b)=a'$ are comparable in $L_2'$. Suppose
$\psi(a)<a'$, then we get a contradiction to the minimality  of
$a'$. Similarly, we can not have $a'< \psi(a)$. Hence $\psi(a)=a'$
in this case also. Hence we are done.
\end{proof}

We, now, conclude the paper by  solving  Isomorphism Problem for the class of lower dismantlable lattices.
\begin{proof}[\bf{Proof of Theorem \ref{t1}}]
It is clear that, if the lattices are isomorphic, then the
zero-divisor graphs are isomorphic. Conversely, suppose that the
zero-divisor graphs of lower dismantlable lattices are isomorphic. By Theorem \ref{703}, there
exists an isomorphism $\phi:G_{\{0\}}(L_1)\rightarrow
G_{\{0\}}(L_2)$  such that $a$ is an adjunct element of $L_1$ if
and only if $\phi (a)$ is an adjunct element of $L_2$. Now, by
Theorem \ref{706},
 the lattices $L_1$ and $L_2$
are isomorphic.
\end{proof}
 In view of Theorem \ref{rp} and Remark \ref{r2}, we have the following corollary.
\begin{cor}Let $T_1$ and $T_2$ be two rooted trees. Then $G(T_1)\cong G(T_2)$ if and only if $T_1\cong T_2$.
\end{cor}
\section{Concluding Remarks} The Isomorphism Problem is one of the  central problems of the zero-divisor graphs of algebraic structures as well as of ordered structures. LaGrange \cite{11} proved that zero-divisor graphs of Boolean rings are isomorphic if and only if the rings are isomorphic. This result was extended by Mohammadian \cite{14} for reduced rings with certain conditions. Recently, Joshi and Khiste \cite{8} proved that the zero-divisor graphs of Boolean posets (Boolean lattices) are isomorphic if and only if  the  Boolean posets (Boolean lattices) are isomorphic. 

It is well known that a Boolean ring can be uniquely determined by a Boolean lattice. Hence Joshi and Khiste \cite{8} extended the result of LaGrange \cite{11} to more general structures, namely Boolean posets. It is easy to observe that every Boolean lattice is an $SSC$ lattice. Joshi et al. \cite{jwp1} proved the Isomorphism Problem for $SSC$-meet semilattices. We feel that with minor modifications, this result is true for $SSC$ posets also.  Thus in nutshell, Joshi et al. \cite{jwp1} extended the result of Joshi and Khiste \cite{8}. From Theorem \ref{ssc}, it is clear that a lower dismantlable lattice $L$ is $SSC$ if and only if the associated basic block of $L$ is $L$ itself. This essentially proves the Isomorphism Problem for the class of lower dismantlable lattices which are basic blocks of itself. This motivate us to prove the Isomorphism Problem for the class of lower dismantlable lattices. Note that neither the class of $SSC$ lattices nor the class of lower dismantlable lattices are contained in each other.

Mohammadian \cite{14} proved the Isomorphism Problem for reduced rings. If $R$ is a reduced commutative ring with unity  then the ideal lattice, the set of all ideals of $R$, is a $0$-distributive lattice, see Joshi and Sarode \cite[Lemma 2.8]{vjs}. Hence one may expect that the Isomorphism Problem may be true for $0$-distributive lattices/posets.  Hence, we raise the following problem.

\begin{prob}
Let $\mathcal{L}$ be the class of  $0$-distributive lattices such that the set of nonzero zero-divisors of $L\in \mathcal{L}$ is $L\setminus\{0,1\}$. 	Is Isomorphism Problem true for the class $\mathcal{L}$?
\end{prob}

\noindent
\textbf{Acknowledgments:} The authors are grateful to the referee for many fruitful suggestions which improved the presentation of the paper. The first author is financially supported by University Grant Commission, New Delhi,
India via minor research project File No. 47-884/14(WRO).

\end{document}